\documentclass{llncs}

	\usepackage{amsmath,amssymb,url,graphicx,rotating,ifthen,algorithm,algorithmic,graphicx,epsfig,multirow,subfigure,array,caption,color,enumerate}
	\usepackage{ifthen}

	\def\E{{\mathbb E}}
	\def\P{{\mathbb P}}
	
	\def\chi{{\mathbf 1}}

	\def\w{\omega }
	\def\R{{\mathbb R}}

	\def\CQFD{\fbox\\}

\begin{document}

	\title{Noisy Optimization: Convergence with a Fixed Number of Resamplings}
	\author{
        Marie-Liesse Cauwet
	}
	\institute{TAO (Inria), LRI, UMR 8623 (CNRS - Univ. Paris-Sud), France}
	\maketitle

	\begin{abstract}
It is known that evolution strategies in continuous domains might not converge in the presence of noise \cite{stocopti5,augerupu}. It is also known that, under mild assumptions, and using an increasing number of resamplings, one can mitigate the effect of additive noise \cite{bignoise2} and recover convergence. We show new sufficient conditions for the convergence of an evolutionary algorithm with constant number of resamplings; in particular, we get fast rates (log-linear convergence) provided that the variance decreases around the optimum slightly faster than in the so-called  multiplicative noise model.

	\end{abstract}

	\section{Introduction}
Given a domain $\mathcal{D}\in \R^d$, with $d$ a positive integer, a noisy objective function is a stochastic process $f$ : $(x,\w)\mapsto f(x,\w)$ with $x\in\mathcal{D}$ and $\w$ a random variable independently sampled at each call to $f$. Noisy optimization is the search of $x$ such that $\E\left[ f(x,\w)\right]$ is approximately minimum. Throughout the paper, $x^{*}$ denotes the unknown exact optimum, supposed to be unique. 
For any positive integer $n$, $\tilde{x}_{n}$ denotes the search point used in the $n^{th}$ function evaluation.
We here consider black-box noisy optimization, i.e we can have access to $f$ only through calls to a black-box which, on request $x$, (i) randomly samples $\w$ (ii) returns $f(x,\w)$.
Among zero-order methods proposed to solve noisy optimization problems, some of the most usual are evolution strategies; \cite{abinvestigation} has studied the performance of evolution strategies in the presence of noise, and investigated its robustness by tuning the population size of the offspring and the mutation strength. Another approach consists in using resamplings of each individual (averaging multiple resamplings reduces the noise), rather than increasing the population size. Resampling means that, when evaluating $f(x,\w)$, several independent copies $\w_1,\dots,\w_r$ of $\w$ are used (i.e. the black-box oracle is called several times with a same $x$) and we use as an approximate fitness value $\frac1r \sum_{i=1}^r f(x,\w_i)$ in the optimization algorithm. The key point is how to choose $r$, number of resamplings, for a given $x$.
Another crucial point is the model of noise. Different models of noise can be considered: additive noise (Eq. \ref{additive}), multiplicative noise (Eq. \ref{multip}) or a more general model (Eq. \ref{ourmodel}). Notice that, in Eq. \ref{ourmodel} when $z>0$, the noise decreases to zero near the optimum; this setting is not artificial as we can observe this behavior in many real problems. 

Let us give an example in which the noise variance decreases to zero around the optimum. Consider a Direct Policy Search problem, i.e. the optimization of a parametric policy on simulations. Assume that we optimize the success rate of a policy. Assume that the optimum policy has a success rate 100\%. Then, the variance is zero at the optimum.

\subsection{Convergence Rates: $\log$-linear convergence and $\log$-$\log$ convergence}
Depending on the specific class of optimization problems and on some internal properties of the algorithm considered, we obtain different uniform rates of convergence (where the convergence can be almost sure, in probability or in expectation, depending on the setting); a fast rate will be a $\log$-linear convergence, as follows: 

\begin{eqnarray}
\mbox{{\bf{Fast rate:}} }\limsup_{n} \frac{\log{|| \tilde{x}_{n}-x^{*}||}}{{n}} = -A\ <\ 0,\label{eqnf}\label{loglin}
\end{eqnarray}
In the noise-free case, evolution strategies typically converge linearly in $\log$-linear scale, as shown in \cite{TCSAnne04-corr,AJT,bey01,Rechenberg,fournierppsn}.\\ 
The algorithm presents a slower rate of convergence in case of $\log$-$\log$ convergence, as follows:

\begin{eqnarray}
\mbox{{\bf{Slow rate:}} }\limsup_{n} \frac{\log{|| \tilde{x}_{n}-x^{*}||}}{{\log n}} = -A\ < \ 0,\label{eqn}\label{loglog}
\end{eqnarray}
The $\log$-$\log$ rates are typical rates in the noisy case (see \cite{abnoise,bignoise2,chen1988,clop,fabian,shamir,decock}). Nevertheless, we will here show that, under specific assumptions on the noise (if the noise around the optimum decreases ``quickly enough'', see section \ref{sectionmodel}), we can reach faster rates: $\log$-linear convergence rates as in Eq. \ref{loglin},  by averaging a constant number of resamplings of $f(x,\w)$.    

\subsection{Additive noise model}

Additive noise refers to:
\begin{equation}
f(x,\w)=||x-x^*||^p+ noise_{\w},\label{additive}
\end{equation}
where $p$ is a positive integer and where $noise_{\w}$ is sampled independently with a fixed given distribution.
In this model, the noise has lower bounded variance, even in the neighborhood of the optimum. The uniform rate typically converges linearly in $\log-\log$ scale (cf Eq. \ref{loglog}) as discussed in \cite{abnoise,chen1988,clop,fabian,shamir,decock}. This important case in applications has been studied in \cite{chen1988,fabian,fabian2,shamir} where tight bounds have been shown for stochastic gradient algorithms using finite differences. When using evolution strategies, \cite{bignoise2} has shown mathematically that an exponential number of resamplings (number of resamplings scaling exponentially with the index of iterations) or an adaptive number of resamplings (scaling as a polynomial of the inverse step-size) can both lead to a $\log$-$\log$ convergence rate.

\subsection{Multiplicative noise model}
Multiplicative noise, in the unimodal spherical case, refers to
\begin{equation}
f(x,\w)=||x-x^*||^p+||x-x^*||^p\times noise_{\w}\label{multip}
\end{equation}
and some compositions (by increasing mappings) of this function, where $p$ is a positive integer and where $noise_{\w}$ is sampled independently with a fixed given distribution.
\cite{augerupu} has studied the convergence of evolution strategies in noisy environments with multiplicative noise, and essentially shows that the result depends on the noise distribution: if $noise_{\w}$ is conveniently lower bounded, then some standard $(1+1)$ evolution strategy converges to the optimum; if arbitrarily negative values can be sampled with non-zero probability, then it does not converge.

\subsection{A more general noise model}\label{sectionmodel}
Eqs. \ref{additive} and \ref{multip} are particular cases of a more general noise model:
\begin{equation}
f(x,\w)=||x-x^*||^p + ||x-x^*||^{pz/2}\times noise_{\omega}\label{ourmodel}.
\end{equation}
where $p$ is a positive integer, $z\geq 0$ and $noise_{\w}$ is sampled independently with a fixed given distribution.
Eq. \ref{ourmodel} boils down to Eq. \ref{additive} when $z=0$ and to Eq. \ref{multip} when $z=2$.
We will here obtain fast rates for some larger values of $z$. More precisely, we will show that when $z>2$, we obtain $\log$-linear rates, as in Eq. \ref{loglin}. Incidentally, this shows some tightness (with respect to $z$) of conditions for non-convergence in \cite{augerupu}.

\section{Theoretical analysis}\label{th}

Section \ref{prel} is devoted to some preliminaries.
Section \ref{th1} presents results for constant numbers of resamplings on our generalized noise model (Eq. \ref{ourmodel}) when $z>2$.

\subsection{Preliminary: noise-free case}\label{prel}
Typically, an evolution strategy at iteration $n$:
\begin{itemize}
\item generates $\lambda$ individuals using the current estimate $x_{n-1}$ of the optimum $x^{*}$ and the so-called mutation strength (or step-size) $\sigma_{n-1}$,
\item provides a pair $(x_{n},\sigma_{n})$ where $x_{n}$ is a new estimate of $x^{*}$ and $\sigma_{n}$ is a new mutation strength.
\end{itemize}
From now on, for the sake of notation simplicity, we assume that $x^{*}=0$. 

For some evolution strategies and in the noise-free case, we know (see e.g. Theorem 4 in \cite{TCSAnne04-corr}) that there exists a constant $A$ such that :
\begin{eqnarray}
\frac{\log(||x_n||)}{n} \xrightarrow[n \rightarrow \infty]{a.s}-A\label{convergence}\\
\frac{\log(\sigma_n)}{n} \xrightarrow[n \rightarrow \infty]{a.s} -A\label{convergence2}
\end{eqnarray}
This paper will discuss cases in which an algorithm verifying Eqs. \ref{convergence}, \ref{convergence2} in the noise-free case also verifies them in a noisy setting.

{\bf{Remarks:}} {\em{ In the general case of arbitrary evolution strategies (ES), we don't know if $A$ is positive, but: 
\begin{itemize}
\item in the case of a $(1+1)$-ES with generalized one-fifth success rule, $A>0$ see \cite{oneplusonepaper};
\item in the case of a self-adaptive $(1, \lambda)$-ES with gaussian mutations, the estimate of $A$ by Monte-Carlo simulations is positive \cite{TCSAnne04-corr}. 
\end{itemize}
}}

\begin{property}\label{propprop}
For some $\delta>0$, for any $\alpha$, $\alpha'$ such that $\alpha<A$ and $\alpha'>A$, there exist $C>0$, $C'>0$, $V>0$, $V'>0$, such that with probability at least $1-\delta$
\begin{eqnarray}
\forall n\geq 1, C'\exp({-\alpha'}n)\leq ||x_n||\leq C\exp({-\alpha}n);\label{linearconv}\\
\forall n\geq 1, V'\exp(-\alpha' n)\leq \sigma_n\leq V\exp(-\alpha n).\label{linearconv2}
\end{eqnarray}
\end{property}

\begin{proof}
For any $\alpha<A$, almost surely, $\log(||x_{n}||)\leq -\alpha n$ for $n$ sufficiently large.
So, almost surely, $\sup_{n\geq 1}\log(||x_{n}||)+\alpha n$ is finite.
Consider $V$ the quantile $1-\frac{\delta}{4}$ of $\exp\left( \sup_{n\geq 1}\log(||x_{n}||)+\alpha n\right).$
Then, with probability at least $1-\frac{\delta}{4}$, $\forall n\geq 1, ||x_{n}||\leq V\exp(-\alpha n).$
We can apply the same trick for lower bounding $||x_{n}||$, and upper and lower bounding $\sigma_{n}$, all of them with probability $1-\frac{\delta}{4}$, so that all bounds hold true simultaneously with probability at least $1-\delta$.\CQFD
\end{proof}

\subsection{Noisy case}\label{th1}
The purpose of this Section is to show that if some evolution strategies perform well (linear convergence in the log-linear scale, as in Eqs. \ref{convergence}, \ref{convergence2}), then, just by considering $Y$ resamplings for each fitness evaluation as explained in Alg. \ref{generales}, they will also be fast in the noisy case.

Our theorem holds for any evolution strategy satisfying the following constraints:
\begin{itemize}
\item At each iteration $n$, a search point $x_n$ is defined and $\lambda$ search points are generated and have their fitness values evaluated.
\item The noisy fitness values are averaged over $Y$ (a constant) resamplings.
\item The $j^{th}$ individual evaluated at iteration $n$ is randomly drawn by $x_n+\sigma_n\mathcal{N}_d$ with $\mathcal{N}_d$ a $d$-dimensional standard Gaussian variable.
\end{itemize}
This framework is presented in Alg. \ref{generales}.
\begin{algorithm}
\begin{algorithmic}
\STATE{Initialize $x_0$ and $\sigma_0$.}
\STATE{$n\leftarrow 1$}
\WHILE{not finished}
	\FOR{$i\in\{1,\dots,\lambda\}$}
		\STATE{Define $x_{n,i}=x_n+\sigma_n \mathcal{N}_d$.}
		\STATE{Define $y_{n,i}=\frac1Y \sum_{k=1}^Y f(x_{n,i},\w_k)$.}
	\ENDFOR
\STATE{Update: $(x_{n+1},\sigma_{n+1})\leftarrow$update($x_{n,1},\dots,x_{n,\lambda}$,$y_{n,1},\dots,y_{n,\lambda},\sigma_{n}$).}
\STATE{$n\leftarrow n+1$}
\ENDWHILE
\end{algorithmic}
\caption{\label{generales} A general framework for evolution strategies. For simplicity, it does not cover all evolution strategies, e.g. mutations of step-sizes as in self-adaptive algorithms are not covered; yet, our proof can be extended to a more general case ($x_{n,i}$ distributed as $x_n+\sigma N$ for some noise $N$ with exponentially decreasing tail). The case $Y=1$ is the case without resampling. Our theorem basically shows that if such an algorithm converges linearly (in log-linear scale) in the noise-free case then the version with $Y$ large enough converges linearly in the noisy case when $z>2$.}
\end{algorithm}

We now state our theorem, under $\log$-linear convergence assumption (cf assumption (\ref{assii}) below).
\begin{theorem}\label{theotheo}
Consider the following assumptions: 
\begin{enumerate}[(i)]
\item the fitness function $f$ satisfies $\E\left[f(x,\w)\right]=\|x\|^p$ and has a limited variance: 
\begin{equation}
Var(f(x,\omega)) \leq \left(\E \left[f(x,\omega)\right]\right)^{z}\mbox{ for some }z>2;\label{limvar}
\end{equation}
\item\label{assii} in the noise-free case, the ES with population size $\lambda$ under consideration is log-linearly converging, i.e. for any $\delta > 0$, for some $\alpha>0$, $\alpha'>0$, there exist $C>0$, $C'>0$, $V>0$, $V'>0$, such that with probability 1-$\delta$, Eqs. \ref{linearconv} and \ref{linearconv2} hold; 
\item the number $Y$ of resamplings per individual is constant. 
\end{enumerate}
Then, if $z>\max\left(\frac{ 2(   p \alpha'-(\alpha-\alpha') d)}{ p \alpha},  \frac{2(   2 \alpha'-\alpha)}{\alpha}\right)$, for any $\delta>0$, there is $Y_0>0$ such that for any $Y\geq Y_0$,  Eqs. \ref{linearconv} and \ref{linearconv2} also hold with probability at least $(1-\delta)^2$  in the noisy case.
\end{theorem}

\begin{corollary}\label{coco}
Under the same assumptions, with probability at least $(1-\delta)^2$,
\begin{equation}
\underset{n}{\lim\sup} \frac{log(||\tilde{x}_{n}||)}{n}\leq-\frac{\alpha}{\lambda Y}\nonumber
\end{equation}
\end{corollary}

{\bf{Proof of Corollary \ref{coco} : }} Immediate consequence of Theorem \ref{theotheo}, by applying Eq. \ref{linearconv} and using $\underset{n}{\lim\sup} \frac{log(||\tilde{x}_{n}||)}{n}=\underset{n}{\lim\sup} \frac{log(||{x}_{n}||)}{\lambda Y n}$.\CQFD

{\bf{Remarks:}} {\em{
	 \begin{itemize}
\item {\bf{Interpretation:}} Informally speaking, our theorem shows that if an algorithm converges in the noise-free case, then it also converges in the noisy case with the resampling rule, at least if $z$ and $Y$ are large enough.
\item Notice that we can choose constants $\alpha$ and $\alpha^{'}$ very close to each other. Then the assumption $z>\max\left(\frac{ 2(   p \alpha'-(\alpha-\alpha') d)}{ p \alpha},  \frac{2(   2 \alpha'-\alpha)}{\alpha}\right)$ boils down to $z>2$.
\item  We show a log-linear convergence rate as in the noise-free case. This means that we get $\log ||\tilde{x}_{n}||$ linear in the number of function evaluations. This is as Eq. \ref{loglin}, and faster than Eq. \ref{loglog} which is typical for noisy optimization with constant variance.
\item In the previous hypothesis, the new individuals are drawn following $x_n+\sigma_n\mathcal{N}_d$ with $\mathcal{N}_d$ a $d$-dimensional standard Gaussian variable, but we could substitute $\mathcal{N}_d$ for any random variable with an exponentially decreasing tail.
\end{itemize} }}

{\bf{Proof of Theorem \ref{theotheo} : }} In all the proof, ${\mathcal{N}_{k}}$ denotes a standard normal random variable in dimension $k$.

{\bf{Sketch of proof:}} Consider an arbitrary $\delta>0$ and $\delta_n=\exp(-\gamma n)$ for some  $n\geq 1$ and $\gamma>0$.\\
We compute in Lemma \ref{boundprobatwopoints} the probability that at least two generated points $x_{n,i_1}$ and $x_{n,i_2}$ at iteration $n$ are ``close'', i.e are such that $|\ ||x_{n,i_1}||^p-||x_{n,i_2}||^p\ |\leq \delta_n$; then we calculate the probability that the noise of at least one of the $\lambda$ evaluated individuals of iteration $n$ is bigger than $\frac{\delta_n}{2}$ in Lemma \ref{probabilityonnoise}. Thus, we can conclude in Lemma \ref{probamisranking} by estimating the probability that at least two individuals are misranked due to noise.
\newline
We first begin by showing a technical lemma.

\begin{lemma}\label{technicallemma}
Let $u\in \mathbb{R}^{d}$ be a unit vector and $\mathcal{N}_{d}$ a $d$-dimensional standard normal random variable. Then for $S>0$ and $\ell >0$, there exists a constant $M>0$ such that :
 \begin{equation*}
\max_{v\geq0} \P(|\ ||u+S{\mathcal{N}_{d}}||^{p}-v |\leq \ell)\leq MS^{-d}\max\left(\ell,\ell^{d/p}\right).
\end{equation*}
\end{lemma}

\begin{proof}
For any $v\geq \ell$, we denote $E_{v\geq\ell}$ the set :
\begin{equation*}
 E_{v\geq\ell}=\left\lbrace x\ ; |\ ||x||^{p}-v\ |\leq \ell\right\rbrace =\left\lbrace x\ ; \left(v-\ell\right)^{\frac{1}{p}}\leq||x||\leq\left(v+\ell\right)^{\frac{1}{p}}\right\rbrace.
\end{equation*}
We first compute $\mu(E_{v\geq\ell})$, the Lebesgue measure of $E_{v\geq\ell}$ :
\begin{equation*} 
\mu(E_{v\geq\ell})=K_d\left\lbrace\left(v+\ell\right)^{\frac{d}{p}}-\left(v-\ell\right)^{\frac{d}{p}}\right\rbrace,\\
\end{equation*}
with $K_d=\frac{(2\pi)^{d/2}}{2\times 4\times \dots \times d}$ if $d$ is even, and $K_d=\frac{2(2\pi)^{(d-1)/2}}{1\times 3\times \dots \times d}$ otherwise. Hence, by Taylor expansion,
$\mu(E_{v\geq\ell}) \leq Kv^{\frac{d}{p} -1}\ell$, where $K=K_{d}\left(2\frac{d}{p}+\underset{v\geq \ell}{\sup}\ \underset{0<\zeta<\frac{\ell}{v}}{\sup}\frac{q''(\zeta)}{2}\frac{\ell}{v}\right)$, with $q(x)=(1+x)^{\frac{d}{p}}$.\\
$\bullet$ If $v\geq \ell$: 
\begin{eqnarray*}
\P(|\ ||u+S{\mathcal{N}_{d}}||^{p}-v |\leq \ell)&=&\P(u+S\mathcal{N}_{d}\in E_{v\geq\ell}),\\
  									&\leq&S^{-d}\underset{x\in E_{v\geq\ell}}{\sup} \left(\frac{1}{\sqrt{2\pi}}\exp(-\frac{||S^{-1}(x-u)||^2}{2})\right)\mu(E_{v\geq\ell}),\\
									&\leq&M_{1}S^{-d} \ell,\\
									&\leq&M_{1}S^{-d}\max\left(\ell,\ell^{d/p}\right).
\end{eqnarray*}
where $M_{1}=\frac{K}{\sqrt{2\pi}}\underset{v\geq\ell}{\sup}\underset{x : ||x||\leq (v+\ell)^{\frac{1}{p}}}{\sup}\left[ v^{\frac{d}{p} -1}\exp\left(-\frac{||S^{-1}(x-u)||^2}{2}\right)\right].$
\begin{eqnarray*}
\mbox{$\bullet$ If $v< \ell$, }& &\P(|\ ||u+S\mathcal{N}_{d}||^{p}-v |\leq \ell)\leq M_{2}S^{-d} \ell^{d/p} \leq M_{2}S^{-d} \max\left(\ell,\ell^{d/p}\right),
\end{eqnarray*}
where $M_{2}=2^{\frac{d}{p}}\frac{K_{d}}{\sqrt{2\pi}}$. Hence the result follows by taking $M=\max(M_{1},M_{2})$.
\CQFD
\end{proof}

\begin{lemma}\label{boundprobatwopoints}
Let us denote by $P^{(1)}_{n}$ the probability that, at iteration $n$, there exist at least two points $x_{n,i_1}$ and $x_{n,i_2}$ such that $|\ ||x_{n,i_1}||^p-||x_{n,i_2}||^p\ |\leq \delta_n$. Then
\begin{equation*}
P^{(1)}_{n}\leq B\lambda^{2}\exp(-\gamma'n),
\end{equation*}
for some $B>0$ and $\gamma'>0$ depending on $\gamma$, $d$, $p$, $C$, $C'$, $V$, $\alpha$, $\alpha'$.
\end{lemma}

\begin{proof}
Let us first compute the probability $P^{(0)}_n$ that, at iteration $n$, two given generated points $x_{n,i_{1}}$ and $x_{n,i_{2}}$ are such that $|\ ||x_{n,i_1}||^p-||x_{n,i_2}||^p\ |\leq \delta_n$. Let us denote by ${\mathcal{N}}_{d}^{1}$ and ${\mathcal{N}}_{d}^{2}$ two $d$-dimensional standard independent random variables, $u\in \mathbb{R}^{d}$ a unit vector and $S_{n}=\frac{\sigma_{n}}{||x_{n}||}$.
\begin{eqnarray*}
P^{(0)}_{n}&=&\P\left(|\ ||x_{n}+\sigma_{n}{\mathcal{N}}_{d}^{1}||^p-||x_{n}+\sigma_{n}{\mathcal{N}}_{d}^{2}||^p\ |\leq \delta_n\right),\\
		&=&\P\left(|\ ||u+S_{n}{\mathcal{N}}_{d}^{1}||^p-||u+S_{n}{\mathcal{N}}_{d}^{2}||^p\ |\leq \frac{\delta_n}{||x_n||^p}\right),\\
     		&\leq&\max_{v\geq0} \P\left(|\ ||u+S_{n}{\mathcal{N}}_{d}^{1}||^{p}-v |\leq \frac{\delta_n}{||x_n||^p}\right).
\end{eqnarray*}
Hence, by Lemma \ref{technicallemma}, there exists a $M>0$ such that $P^{(0)}_{n}\leq MS_{n}^{-d}\left(\frac{\delta_{n}}{||x_n||^{p}}\right)^{m}$, where $m$ is such that $\left(\frac{\delta_{n}}{||x_n||^{p}}\right)^{m}=\max\left(\frac{\delta_{n}}{||x_{n}||^{p}},\left(\frac{\delta_{n}}{||x_{n}||^{p}}\right)^{d/p}\right)$. Moreover $S_{n}\geq V'C^{-1}\exp(-(\alpha' - \alpha)n)$ by Assumption (\ref{assii}). Thus $P^{(0)}_{n}\leq B\exp(-\gamma' n)$, with $B=MV'^{-d}C^{d}C'^{-mp}$ and $\gamma'=d(\alpha-\alpha^{'})+m\gamma-mp\alpha'$. In particular, $\gamma'$ is positive, provided that $\gamma$ is sufficiently large.\\
By union bound, $P^{(1)}_{n}\leq \frac{(\lambda-1)\lambda}{2} P^{(0)}_n\leq B\lambda^{2}\exp(-\gamma'n)$.
\CQFD
\end{proof}

We now provide a bound on the probability $P^{(3)}_{n}$ that the fitness value of at least one search point generated at iteration $n$ has noise (i.e. deviation from expected value) bigger than $\frac{\delta_{n}}{2}$ in spite of the $Y$ resamplings.
\begin{lemma}\label{probabilityonnoise}
$$P^{(3)}_{n}:=\P\left(\exists i\in \{ 1,\dots,\lambda\}\ ;\ \left|\frac{1}{Y}\sum_{j=1}^{Y} f(x_{n,i},\omega_{j})- \E\left[f(x_{n,i},\omega_{j})\right] \right|\geq \frac{\delta_n}{2}\right)$$
$$\leq \lambda B'\exp(-\gamma''n)$$ for some $B'>0$ and $\gamma''>0$ depending on $\gamma$, $d$, $p$, $z$, $C$, $Y$, $\alpha$, $\alpha'$.\\ 
\end{lemma}

\begin{proof}
First, for one point $x_{n,i_{0}}$, $i_{0}\in \{1,\dots,\lambda\}$ generated at iteration $n$, we write $P^{(2)}_{n}$ the probability that when evaluating the fitness function at this point, we make a mistake bigger than $\frac{\delta_{n}}{2}$.\\ 
$P^{(2)}_{n}=\P(|\frac{1}{Y}\sum_{j=1}^{Y} f(x_{n,i_{0}},\omega_{j})- \E\left[f(x_{n,i_{0}},\omega_{j})\right]|\geq \frac{\delta_n}{2})\leq B'\exp(-\gamma''n)$ by using Chebyshev's inequality, where $B'=4Y^{-1}C^{pz}$ and $\gamma''=\alpha zp-2\gamma$. In particular, $\gamma'' >0$ if $z>\frac{2 (mp \alpha'-(\alpha-\alpha') d) }{p \alpha m}$; hence, if $z\geq \max\left(\frac{ 2(   p \alpha'-(\alpha-\alpha') d)}{ p \alpha},  \frac{2(   2 \alpha'-\alpha)}{\alpha}\right)$, we get $\gamma''>0$.\\ 
Then, $P^{(3)}_{n}\leq \lambda P^{(2)}_n$ by union bound.\CQFD
\end{proof}

\begin{lemma}\label{probamisranking}
Let us denote by $P_{misranking}$ the probability that in at least one iteration, there is at least one misranking of two individuals.
Then, if $z> \max\left(\frac{ 2(   p \alpha'-(\alpha-\alpha') d)}{ p \alpha},  \frac{2(   2 \alpha'-\alpha)}{\alpha}\right)$ and $Y$ is large enough, 
$P_{misranking}\leq \delta$. 
\end{lemma}

This lemma implies that with probability at least $1-\delta$, provided that $Y$ has been chosen large enough, we get the same rankings of points as in the noise free case. In the noise free case Eqs. \ref{linearconv} and \ref{linearconv2} hold with probility at least $1-\delta$ - this proves the convergence with probability at least  $(1-\delta)^2$, hence the expected result; the proof of the theorem is complete.\CQFD

\begin{proof} (of the lemma)

We consider the probability $P_{n}^{(4)}$ that two individuals $x_{n,i_1}$ and $x_{n,i_2}$ at iteration $n$ are misranked due to noise, so 
\begin{eqnarray}
||x_{n,i_1}||^p&\leq& ||x_{n,i_2}||^p\label{eq1}\\
\mbox{ and }\frac{1}{Y}\sum_{j=1}^{Y} f(x_{n,i_{1}},\w_{j})&\geq& \frac{1}{Y}\sum_{j=1}^{Y} f(x_{n,i_{2}},\omega_{j})\label{eq2}
\end{eqnarray}
 
Eqs. \ref{eq1} and \ref{eq2} occur simultaneously if either two points have very similar fitness (difference less than $\delta_n$) or the noise is big (larger than $\frac{\delta_n}{2}$). Therefore,
$P^{(4)}_{n}\leq P^{(1)}_{n}+P^{(3)}_{n}\leq \lambda^{2} P^{(0)}_{n}+\lambda P^{(2)}_{n}\leq (B+B')\lambda^{2}\exp(-\min(\gamma',\gamma'')n)$.\\
$P_{misranking}$ is upper bounded by $\sum_{n\geq 1} P^{(4)}_{n}<\delta$ if $\gamma'$ and $\gamma''$ are positive and constants large enough. $\gamma'$ and $\gamma''$ can be chosen positive simultaneously if $z> \max\left(\frac{ 2(   p \alpha'-(\alpha-\alpha') d)}{ p \alpha},  \frac{2(   2 \alpha'-\alpha)}{\alpha}\right)$.
\CQFD
\end{proof}

\section{Experiments : how to choose the right number of resampling ?}

We consider in our experiments a version of multi-membered evolution strategies, the ($\mu$,$\lambda$)-ES, where $\mu$ denotes the number of parents and $\lambda$ the number of offspring ($\mu\leq\lambda$; see Alg. \ref{es}). We denote $(x_{n}^{1},\dots,x_{n}^{\mu})$ the $\mu$ parents at iteration $n$ and $(\sigma_{n}^{1},\dots,\sigma_{n}^{\mu})$ their corresponding step-size. At each iteration, a ($\mu$,$\lambda$)-ES noisy algorithm : (i) generates $\lambda$ offspring by mutation on the $\mu$ parents, using the corresponding mutated step-size, (ii) selects the $\mu$ best offspring by ranking the noisy fitness values of the individuals. Thus, the current approximation of the optimum $x^{*}$ at iteration $n$ is $x_{n}^{1}$, to be consistent with the previous notations, we denote $x_{n}=x_{n}^{1}$ and $\sigma_{n}=\sigma_{n}^{1}$.

\begin{algorithm}
	\begin{algorithmic}
		\STATE{{\bf{Parameters : }} $Y> 0$, $\lambda\geq \mu>0$, a dimension $d>0$.}
		\STATE{{\bf{Input : }} $\mu$ initial points $x_{1}^{1},\dots,x_{1}^{\mu} \in \R^{d}$ and initial step size $\sigma_{1}^{1}>0, \dots,\sigma_{1}^{\mu}  >0$.}
		\STATE{$n\leftarrow 1$}
		\WHILE{(true)}
		\STATE{Generate $\lambda$ individuals indenpendently using :}
		\begin{eqnarray*}
		\sigma_{j}&=&\sigma_{n}^{mod(j-1,\mu)+1}\times exp(\frac{1}{2d}\times \mathcal{N}_{1})\\
		i_{j}&=&x_{n}^{mod(j-1,\mu)+1}+\sigma_{j}\mathcal{N}_{d}
		\end{eqnarray*}
		\STATE{$\forall j\in \{1,\dots,\lambda\}$, evaluate $i_{j}$  $Y$ times. Let $y_{j}$ be the averaging over these $Y$ evaluations. }
		\STATE{Define $j_{1},\dots,j_{\lambda}$ so that $y_{j_{1}}\leq y_{j_{2}}\leq \dots \leq y_{j_{\lambda}}$. }
		\STATE{{\bf{Update : }} compute $\sigma_{n+1}^{k} $ and $x_{n+1}^{k}$ for $k\in \{1,\dots,\mu\}$:
				\begin{eqnarray*}
				\sigma_{n+1}^{k}&=&\sigma_{j_{k}}\\
				 x_{n+1}^{k}&=&x_{j_{k}}
				\end{eqnarray*}}
		\STATE{$n\leftarrow n+1$}
		\ENDWHILE
	\end{algorithmic}
	\caption{\label{es}An evolution strategy, with constant number of resamplings. If we consider $Y=1$, we obtain the case without resampling. ${\mathcal{N}_{k}}$ is a $k$-dimensional standard normal random variable.} 
\end{algorithm}
Experiments are performed on the fitness function $f(x,\omega)=||x||^p + ||x||^{pz/2}\mathcal{N}$, with $x\in \R^{15}$, $p=2$, $z=2.1$, $\lambda=4$, $\mu=2$, and $\mathcal{N}$ a standard gaussian random variable, using a budget of $500000$ evaluations. The results presented here are the mean and the median over 50 runs.
The positive results are proved, above, for a given quantile of the results. This explains the good performance in Fig. \ref{med} (median result) as soon as the number of resamplings is enough. The median performance is optimal with just 12 resamplings.
On the other hand, Fig. \ref{moy} shows the mean performance of Alg. \ref{es} with various numbers of resamplings. We see that a limited number of runs diverge so that the mean results are bad even with 16 resamplings; results are optimal (on average) for 20 resamplings. 

Results are safer with 20 resamplings (for the mean), but faster (for the median) with a smaller number of resamplings.
\begin{figure}
\center
\includegraphics[width=0.8\linewidth]{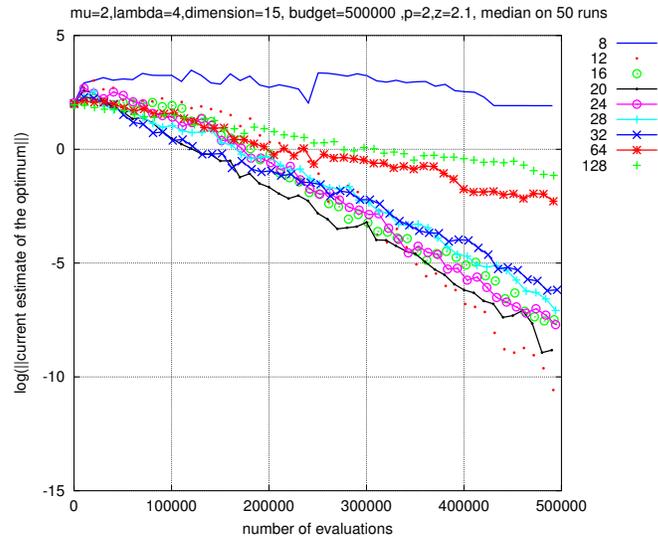}
\caption{\label{med}Convergence of Self-Adaptive Evolution Strategies: Median results.}
\end{figure}
\begin{figure}
\center
\includegraphics[width=0.8\linewidth]{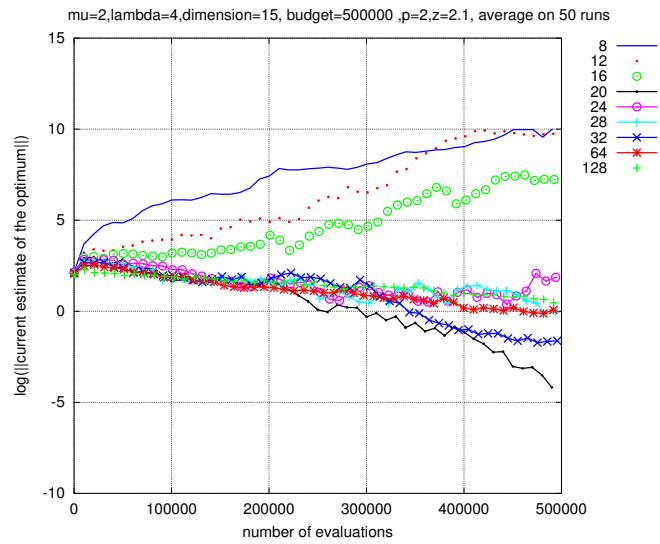}
\caption{\label{moy}Convergence of Self-Adaptive Evolution Strategies: Mean results.}
\end{figure}

\section{Conclusion}
We have shown that applying evolution strategies with a finite number of resamplings when the noise in the function decreases quickly enough near the optimum provides a convergence rate as fast as in the noise-free case. More specifically, if the noise decreases slightly faster than in the multiplicative model of noise, using a constant number of revaluation leads to a log-linear convergence of the algorithm. The limit case of a multiplicative noise has been analyzed in \cite{augerupu}; a fixed number of resamplings is not sufficient for convergence when the noise is unbounded.

{\bf{Further work.}} We did not provide any hint for choosing the number of resamplings. Proofs based on Bernstein races \cite{icmlbadtruc} might be used for adaptively choosing the number of resamplings.
\subsection*{Acknowledgements}
This paper was written during a stay in Ailab, Dong Hwa University, Hualien, Taiwan.

\bibliographystyle{abbrv}
\bibliography{tout,tout2,teytaud}

\begin{thebibliography}{10}

\bibitem{abinvestigation}
D.~Arnold and H.-G. Beyer.
\newblock Investigation of the $(\mu,\lambda)$-es in the presence of noise.
\newblock In {\em Proc. of the IEEE Conference on Evolutionary Computation
  ({CEC} 2001)}, pages 332--339. IEEE, 2001.

\bibitem{abnoise}
D.~Arnold and H.-G. Beyer.
\newblock Local performance of the (1 + 1)-es in a noisy environment.
\newblock {\em Evolutionary Computation, IEEE Transactions on}, 6(1):30 --41,
  feb 2002.

\bibitem{stocopti5}
D.~V. Arnold and H.-G. Beyer.
\newblock A general noise model and its effects on evolution strategy
  performance.
\newblock {\em IEEE Transactions on Evolutionary Computation}, 10(4):380--391,
  2006.

\bibitem{bignoise2}
S.~Astete-Morales, J.~Liu, and O.~Teytaud.
\newblock log-log convergence for noisy optimization.
\newblock In {\em Proceedings of EA 2013}, LLNCS, page accepted. Springer,
  2013.

\bibitem{TCSAnne04-corr}
A.~Auger.
\newblock Convergence results for (1,$\lambda$)-{SA}-{ES} using the theory of
  $\varphi$-irreducible {M}arkov chains.
\newblock {\em Theoretical Computer Science}, 334(1-3):35--69, 2005.

\bibitem{oneplusonepaper}
A.~Auger.
\newblock Linear convergence on positively homogeneous functions of a
  comparison-based step-size adaptive randomized search: the (1+1)-es with
  generalized one-fifth success rule.
\newblock {\em submitted}, 2013.

\bibitem{AJT}
A.~Auger, M.~Jebalia, and O.~Teytaud.
\newblock (x,sigma,eta) : quasi-random mutations for evolution strategies.
\newblock In {\em EA}, page 12p., 2005.

\bibitem{bey01}
H.-G. Beyer.
\newblock {\em The Theory of Evolution Strategies}.
\newblock Natural Computing Series. Springer, Heidelberg, 2001.

\bibitem{chen1988}
H.~Chen.
\newblock {Lower rate of convergence for locating the maximum of a function}.
\newblock {\em {Annals of statistics}}, 16:1330--1334, Sept. 1988.

\bibitem{clop}
R.~Coulom.
\newblock Clop: Confident local optimization for noisy black-box parameter
  tuning.
\newblock In {\em Advances in Computer Games}, pages 146--157. Springer Berlin
  Heidelberg, 2012.

\bibitem{fabian}
V.~Fabian.
\newblock {Stochastic Approximation of Minima with Improved Asymptotic Speed}.
\newblock {\em {Annals of Mathematical statistics}}, 38:191--200, 1967.

\bibitem{fabian2}
V.~Fabian.
\newblock {\em Stochastic Approximation}.
\newblock SLP. Department of Statistics and Probability, Michigan State
  University, 1971.

\bibitem{icmlbadtruc}
V.~Heidrich-Meisner and C.~Igel.
\newblock Hoeffding and bernstein races for selecting policies in evolutionary
  direct policy search.
\newblock In {\em ICML '09: Proceedings of the 26th Annual International
  Conference on Machine Learning}, pages 401--408, New York, NY, USA, 2009.
  ACM.

\bibitem{augerupu}
M.~Jebalia, A.~Auger, and N.~Hansen.
\newblock {Log linear convergence and divergence of the scale-invariant
  (1+1)-ES in noisy environments}.
\newblock {\em Algorithmica}, 2010.

\bibitem{Rechenberg}
I.~Rechenberg.
\newblock {\em Evolutionstrategie: Optimierung Technischer Systeme nach
  Prinzipien des Biologischen Evolution}.
\newblock Fromman-Holzboog Verlag, Stuttgart, 1973.

\bibitem{shamir}
O.~Shamir.
\newblock On the complexity of bandit and derivative-free stochastic convex
  optimization.
\newblock {\em CoRR}, abs/1209.2388, 2012.

\bibitem{decock}
O.~Teytaud and J.~Decock.
\newblock {Noisy Optimization Complexity}.
\newblock In {\em {FOGA - Foundations of Genetic Algorithms XII - 2013}},
  Adelaide, Australie, Feb. 2013.

\bibitem{fournierppsn}
O.~Teytaud and H.~Fournier.
\newblock Lower bounds for evolution strategies using vc-dimension.
\newblock In G.~Rudolph, T.~Jansen, S.~M. Lucas, C.~Poloni, and N.~Beume,
  editors, {\em PPSN}, volume 5199 of {\em Lecture Notes in Computer Science},
  pages 102--111. Springer, 2008.

\end{thebibliography}
  
\end{document}